\documentclass[journal]{IEEEtran}
\usepackage{color,latexsym,amsfonts,amssymb}
\usepackage{amsmath,amsthm}
\usepackage{dsfont}
\usepackage{cases}
\usepackage{amsfonts,amssymb,amsmath,amsthm}
\usepackage{graphicx}
\usepackage{array} 
\usepackage{epstopdf}
\newtheorem{theorem}{Theorem}
\newtheorem{lemma}[theorem]{Lemma}
\newtheorem{remark}{Remark}

\newtheorem{corollary}[theorem]{Corollary}

\newtheorem{assumption}{Assumption}
\newtheorem{definition}{Definition}

\usepackage{tikz}
\newcommand\copyrighttext{%
	\footnotesize \textcopyright 2025 IEEE. Personal use of this material is permitted.
	Permission from IEEE must be obtained for all other uses, in any current or future
	media, including reprinting/republishing this material for advertising or promotional
	purposes, creating new collective works, for resale or redistribution to servers or
	lists, or reuse of any copyrighted component of this work in other works.
	DOI: 10.1109/TAC.2025.3580631}
\newcommand\copyrightnotice{%
	\begin{tikzpicture}[remember picture,overlay]
		\node[anchor=south,yshift=10pt] at (current page.south) 
		{\fbox{\parbox{\dimexpr\textwidth-\fboxsep-\fboxrule\relax}{\copyrighttext}}};
	\end{tikzpicture}%
}

\ifCLASSINFOpdf
\else
\fi
\begin{document}
%
\title{Mean uniformly stable function and its application to almost sure stability analysis of randomly switched time-varying systems}
%
%
%

\author{Qian~Liu,
        Yong~He,~\IEEEmembership{Senior Member,~IEEE,}
        and~Lin~Jiang,~\IEEEmembership{Member,~IEEE} \vspace{-1cm}
\thanks{This work was supported by the National Natural Science Foundation of China under Grant 62433018 and the 111 Project under Grant B17040. (\textit{Corresponding author: Yong He})}
\thanks{Qian Liu and Yong He are with School of Automation, China University of Geosciences, Wuhan430074, P. R. China, (e-mail: heyong08@cug.edu.cn).}
\thanks{Lin Jiang is with Department of Electrical Engineering and Electronic,
	The University of Liverpool, Liverpool L69 3GJ, United Kingdom, (e-mail:
	ljiang@liv.ac.uk).}
}

\maketitle
\copyrightnotice
\begin{abstract}
This paper investigates uniform almost sure stability of randomly switched time-varying systems. Mode-dependent indefinite multiple Lyapunov functions (iMLFs) are introduced to assess stability properties of diverse time-varying subsystems. To realize the stability conditions establishment based on iMLFs, we present a novel condition so-called mean uniformly stable function for time-varying parameters of iMLFs' derivatives. Our approach provides a probabilistic perspective, making iMLFs well-suited for randomly switched time-varying systems. Moreover, the MUSF condition reveals an essential insight: ensuring that each time-varying subsystem remains mean-bounded during its corresponding sojourn time interval is a prerequisite for the almost sure stability of the entire system. Additionally, the combination of iMLFs and MUSFs is able to accommodate stability analysis scenarios where some subsystems are unstable or exhibit non-exponential decay. Numerical examples are provided to demonstrate the effectiveness and advantages of our approach
\end{abstract}

\vspace{-0.2cm}

%
\IEEEpeerreviewmaketitle

\vspace{-0.4cm}
\section{Introduction}
As an important class of dynamics, switched systems have received extensively investigations over the last few decades. Among these systems, randomly switched systems have garnered increasing attention due to their significance in both theoretical developments of practical systems with random structural changes, such as flight systems, power systems, communication systems, and networked control systems. A randomly switched system consists of a collection of subsystems and a random switching signal that dictates the transitions between subsystems, where semi-Markov chains (including Markov chains) and renewal processes are commonly employed as switching signals (see the detail in \cite{Mao2006,Chatterjee2007C,Chatterjee2007J,Chatterjee2011,Shi2015}).

Stability plays a crucial and fundamental role in the dynamic analysis and control design of systems. The main concepts of stability for randomly switched systems include: stability in probability, $p$th moment stability and almost sure stability \cite{Mao2006,Chatterjee2011,Shi2015}. Stability in probability characterizes the convergence behaviors of states through a probability measure, requiring that the events of stable state trajectories occur with a nonzero probability. In contrast, the $p$th moment stability emphasizes the average behavior of trajectories, representing stable states through mathematical expectation. This stability concept simplifies stochastic analysis problems into deterministic ones, making it widely applicable in the analysis and synthesis of stochastic systems \cite{Huang2009,Impicciatore2020sufficient,Chen2021mean}. Notably, $p$th moment stability can imply the stability in probability, some authors even express stability in probability via integral of the mean square ($2$th moment) of the states instead of probability measure (see Definition 2.1 of \cite{Feng1992stochastic} and Definition 1 of \cite{Chen2021sampled}). Beyond these two concepts, almost sure stability offers a more practical perspective: it requires that the events of the stable state trajectory occur with probability 1. As pointed out by \cite{Feng1992stochastic}, almost sure stability is of significant practical importance, as it is the sample paths rather than the $p$th moments of the state process that are observed in practice.

Given its practical utility in various stochastic systems, almost sure stability serves as the primary focus of investigation in this work. For the linear Markov switched time-invariant systems, \cite{Feng1992stochastic} has confirmed that $p$th moment stability implies almost sure stability but the reverse is not true. For the nonlinear randomly switched systems, there is no obvious inclusion relationship between these two stability concepts \cite{Chatterjee2011,Guo2018stability}. In this work, we aim to investigate uniformly almost surely stable property of nonlinear randomly switched systems, considering all the subsystems are time-varying. In doing so, we will explore what stochastic properties are satisfied by the subsystems to ensure the almost sure stability of the whole system.

To get stability estimate of each time-varying subsystem, we utilize a mode-dependent indefinite Lyapunov functions (iMLFs) scheme: for $i$th subsystem, there holds \vspace{-0.2cm}
	\begin{equation}\label{iMLF1}
		\dot{V}_i(t,x(t))\leq\lambda_i(t)V_i(t,x(t)),
	\end{equation}
where $\dot{V}_i(t,x)=\frac{\partial V_i(t,x)}{\partial t}+\frac{\partial V_i(t,x)}{\partial x}f_i(t,x)$.
The prototype of this form can be traced back to the research of time-varying systems in Krasovskii's monograph \cite{Krasovskii1963} in 1959. In Theorem 10.3 of \cite{Krasovskii1963}, Krasovskii explored the idea of replacing the negative defined constant parameter of the time derivative of Lyapunov function with a time-varying function, i.e., 
	\begin{equation}\label{iLF}
		\dot{V}(t,x(t))\leq\lambda(t)V(t,x(t)).
	\end{equation}
This approach introduces a time-varying parameter in the derivative of the Lyapunov function to capture the time-varying decay rate of energy in time-varying systems.
It is noteworthy that the time-varying parameter $\lambda(t)$ in \eqref{iLF} is allowed to be indefinitely signed during some time intervals, i.e., $\lambda(t)\in\mathbb{R}$, and thus has been named by some researchers as indefinite Lyapunov functions method (iLFs) \cite{Ning2012,Ning2014,Ning2015,Ning2018indefinite}. In recent decades, the iLFs have been applied to various types of time-varying systems, such as those with impulsive and Markov switching \cite{Wu2006}, impulsive systems \cite{Peng2010SPL,Peng2017new},  stochastic hybrid systems \cite{Peng2010IEEE} and so on.

Since randomly switched systems share same mathematical description with deterministic switched systems except switching signals, one can adopt multiple Lyapunov functions (MLFs) method \cite{Branicky1998} to conduct stability analysis. A common form of MLFs assumes that for $i$th subsystem there holds: \vspace{-0.2cm}
	\begin{equation}\label{MLF}
		\dot{V}_i(x(t))\leq c_{i}V_i(x(t)),\ c_i\in\mathbb{R}.
	\end{equation}
For randomly switched systems, its outstanding advantage is describing different energy decay rates of subsystems, even when some of them are unstable \cite{Leth2013,Schioler2014,Wang2018}. For switched time-varying systems, it's natural to combine MLFs \eqref{MLF} and iLFs \eqref{iLF} to investigate stability \cite{Lu2018stabilizability,Long2018integral,Zhang2019new,Chen2022improved,Liu2021indefinite,Liu2023improved}. However, in the literature mentioned above iMLFs method assumes derivatives of Lyapunov functions share a common time-varying parameter, that is, \vspace{-0.2cm}
	\begin{equation}\label{iMLF}
		\dot{V}_i(t,x(t))\leq\lambda(t)V_i(t,x(t)).
	\end{equation}
	Unlike MLFs \eqref{MLF}, this form can not describe different decay rates of subsystems. Hence, mode-dependent iMLFs method is presented in \cite{Chen2017relaxed,Wu2018}. It assumes that each subsystem owns a time-varying stability estimator like \eqref{iMLF1}, which is a flexible and efficient tool for stability analysis of deterministic switched time-varying systems \cite{Chen2017relaxed,Wu2018}.

	The crucial part of establish stability criteria through iMLFs \eqref{iMLF1} is the treatment of time-varying functions $\lambda_i(t)$, where the ultimate goal is to make the series formed by iMLFs along the switching time points converges to $0$ by restraining $\lambda_i(t)$. In \cite{Chen2017relaxed}, global constraints on $\lambda_{\sigma(t)}(t)$ are introduced in relation to switching signals $\sigma(t)$ with the average dwell time condition. For any switching dwell time interval $[t_k,t_{k+1}),\ k=0,1,2,\cdots$, there holds \vspace{-0.2cm}
	\begin{equation}\label{ASF_lr}
		\int_{t_0}^{\infty}\lambda_{\sigma(h)}(h)\mathrm{d}h=\lim_{k\rightarrow\infty}\sum_{l=0}^k\int_{t_l}^{t_{l+1}}\lambda_{\sigma(t_l)}(h)\mathrm{d}h=-\infty,
	\end{equation}
	which ensures $V_{\sigma}(t,x)$ to converge to $0$ and thus derives stability criteria. To ensure uniformity of attractivity and stability, the following restriction \eqref{USF_lr} is given \vspace{-0.2cm}
	\begin{equation}\label{USF_lr}
		\int_{t_k}^{t}\lambda_{\sigma(t_k)}(h)\mathrm{d}h<M.
	\end{equation}
	In addition, Wu et al. presents a linear constraint on the integral of $\lambda_{\sigma(t)}(t)$ \cite{Wu2018}, that is, \vspace{-0.2cm}
	\begin{align}\label{UESF_lr}
		\int_{t_0}^{t}\lambda_{\sigma(h)}(h)\mathrm{d}h=&\sum_{l=0}^{k-1}\int_{t_l}^{t_{l+1}}\lambda_{\sigma(t_l)}(h)\mathrm{d}h+\int_{t_k}^{t}\lambda_{\sigma(t_k)}(h)\mathrm{d}h\notag\\
		\leq&-a(t-t_0)+b.
	\end{align}
	As can be seen in \eqref{ASF_lr}-\eqref{UESF_lr}, switching moments $t_k$ should be pre-determined according to switching signal. Unfortunately, randomly switching signals  have random switching time points. Therefore, $t_k$ cannot be foreknown such that the conditions \eqref{ASF_lr}-\eqref{UESF_lr} cannot be verified in randomly switched time-varying systems. 

	To the best of the author's knowledge, there are few results regarding the use of iMLFs \eqref{iMLF1} to establish sufficient conditions for the almost sure stability of randomly switched time-varying systems.
	The challenges in applying iMLFs arise not only from the unpredictable switching time points, but also from determining the stability conditions through $\lambda_{i}(t)$ that the subsystems must satisfy to ensure the almost sure stability of the whole system. To address these challenges, we present a constraint on $\lambda_{i}(t)$ from the perspective of a random process, referred to as the mean uniformly stable function (MUSF). 
	The key contributions of the MUSF condition are as follows:
	\begin{itemize}
		\item
		This novel constraint requires integrals of $\lambda_{i}(t)$ between every two successive switching points are mean bounded. Our evaluation actually is based on the expected length of the interval confined instead of its real length, without the need to know switching time points $t_k$ and $t_{k+1}$. 
		\item
		The integral of $\lambda_{i}(t)$ is bounded solely by the sojourn time interval, ensuring that the MUSF remains uniform over time. This uniformity, in turn, guarantees the uniformity of both stability and attractivity of system.
		\item
		The MUSFs condition provides a flexible framework for stability analysis. It indicates that each subsystem can remain mean bounded during its sojourn time, even if some subsystems are unstable. In such cases, $\lambda_{i}(t)$ can take positive values, enhancing the flexibility of iMLFs for stability estimation.
		\item 
		The MUSFs condition reveals that a sufficient condition for almost sure stability is that each subsystem is mean uniformly bounded. Importantly, for unstable subsystems, the MUSFs condition still ensures that they are constrained to be mean bounded with respect to the corresponding sojourn time intervals, as previously discussed.
    \end{itemize}

The subsequent sections of this paper are organized as follows. Section 2 will formulate the system model and introduce definition of almost sure stability and useful lemmas. In Section 3, we will present an example to illustrate the idea of MUSF condition. Then, we will introduce notions of iMLFs and MUSFs. Based on these preparations, we derive sufficient conditions of uniform almost sure stability for randomly switched time-varying systems. In Section 4, numerical examples will be given to illustrate the effectiveness and advantages of our results.

\textit{Notations:} Let $(\Omega, \mathcal{F}, \mathcal{F}_{t}, \mathbb{P})$ be a complete probability space.
$\mathbb{R}^{n\times m}$ denotes the $n\times m$-dimensional Euclidean space, $\mathbb{R}^+=[0,+\infty)$, $\mathbb{R}_{t_0}^+=[t_0,+\infty)$. $\left| \cdot \right|$ denotes the Euclidean norm. $\mathbb{N}^+$ denotes nonnegative integers.
$\mathcal{K}$ denotes the class of strictly increasing continuous functions, $\mathcal{K}_{\infty}$ is the subset of $\mathcal{K}$ where functions trend to infinity as $t\rightarrow\infty$. $\mathcal{C}^{1,2}(\mathbb{R}^+_{t_0},\mathbb{R}^n;\mathbb{R}^n)$ represents the family of all $\mathbb{R}^n$-valued functions whose first variable's derivative is continuous and second variable's second derivative is continuous. Moreover, $\mathcal{PC}(\mathbb{R}^+_{t_0};\mathbb{R})$ represents the family of all piece-wise continuous function.

\vspace{-0.4cm}
\section{Model formulation and preliminaries}
Considering the following time-varying randomly switched systems: \vspace{-0.2cm}
\begin{equation}\label{eq:SDEwMS}
	\begin{cases}
		\dot{x}(t)=f_{r}(t,x(t)),\\
		x(t_0)=\phi,
	\end{cases}
\end{equation}
where $x(t)\in\mathbb{R}^n$ is state, $r(t)\in\mathbb{S}\subset\mathbb{N}^+$ is the switching signal and also a c\`adl\`ag stochastic process, where $\mathbb{S}$ is set of the switch modes. For all $r(t)=i\in\mathbb{S}$, $f_i$ satisfy the locally Lipschitz condition and $f_i(t,0,0)=0$. For simplify, we note $r(t_0)=r_0$.
The set of switching moments is defined by $\mathcal{T}=\{ t_0\leq t_1\leq t_2\leq\cdots\leq t_k\leq\cdots,\ k\in\mathbb{N}^+\}$. $N_i(t,s)$ represents the occurrence counts of $i$th subsystem during $[s,t]$, and $N_{ij}(t,s)$ is the counts of transitions from $i$th to $j$th subsystem and. The counting function $N(t,s)=\sum_{i\in\mathbb{S}}N_i(t,s)=\sum_{i,j\in\mathbb{S}}N_{ij}(t,s)$ is the total switch counts. For $k\in\mathbb{N}^+,\ i,j\in\mathbb{S}$, $S(k)=t_k-t_{k-1}$, $S_{i}(k)$ and $S_{ij}(k)$ are $k$th sojourn time, the $k$th sojourn time of mode $i$, and the sojourn time of mode $j$ at the $k$th transition from mode $i$ to mode $j$, respectively. Over the time interval $[s,t]$, $T_i(t,s)=\sum_{k=1}^{N_i(t,s)}S_i(k)$ and $T_{ij}(t,s)=\sum_{k=1}^{N_{ij}(t,s)}S_{ij}(k)$ are the total sojourn time of mode $i$ and the total sojourn time from mode $i$ to mode $j$, respectively.

In this work, we mainly consider the following randomly switching signals.
\begin{definition}[\cite{Barbu2009semi}]
	The switching signal $r(t)$ is said to be
	\begin{enumerate}
		\item [(1)] 
		semi-Markov switching signal, if for $t\in[t_{k-1},t_k]$ embedded process $r(t)=r_k$ is a discrete-time Markov chain with transition probability matrix $P=[p_{ij}]$, where $p_{ij}:=\mathbb{P}[r_{k+1}=j|r_k=i]$. The sojourn time $S(k)=t_k-t_{k-1}$ is random variable with distribution function \vspace{-0.2cm}
		\begin{equation*}
			F_{ij}(t)\!=\!\mathbb{P}\left[S(k)\!\leq\!t|r_k\!=\!j,r_{k \!-\!1}\!=\!i\right],i,j\in\mathbb{S},t\!\geq\!t_0.
		\end{equation*}
		$r_k$ and $S(k)$ are independent and memory-less i.e. \vspace{-0.2cm}
		\begin{align*}
			\mathbb{P}\left[r_k=i,S(k)\leq t|\cup_{l=1}^{k-1}\left\{r_{l},S(l)\right\}\right]\\
			=\mathbb{P}\left[r_k\!=\!i,S(k)\!\leq\! t|r_{k-1}\right],i\in\mathbb{S},t\geq t_0.
		\end{align*}
		\item [(2)]
		Markov switching signal, if for $t\in[t_{k-1},t_k],\ k\in\mathbb{N}^+$, $r(t)=r_k$ there holds \vspace{-0.2cm}
		\begin{align*}
			\mathbb{P}\left[r_k=i|\cup_{l=1}^{k-1}\left\{r_{l}\right\}\right]=\mathbb{P}\left[r_k=i|r_{k-1}\right],i\in\mathbb{S},t\geq t_0.
		\end{align*}
		And its generator $Q=[q_{ij}]$ is defined by \vspace{-0.2cm}
		\begin{equation*}
			\mathbb{P}\left[r(t+h)=j|r(t)=i\right]=
			\begin{cases}
				q_{ij}h+o(h),\ &i\neq j\\
				1+q_{ij}h+o(h),\ &i=j
			\end{cases}
		\end{equation*}
		where $h>0$, $\lim_{h\rightarrow0}\frac{o(h)}{h}=0$ and $q_{ii}=-\sum_{j\neq i}q_{ij}$.
		\item [(3 )]
		renewal process switching, if the switching counts $N(t,t_0)$ satisfies a renewal process, that is, the sojourn times $\{S(k),k\in\mathbb{N}^+\}$ are a sequence of nonnegative independent random variables with a common distribution $F$.
	\end{enumerate}	
\end{definition}\vspace{-0.3cm}

\begin{lemma}[\cite{Wu2018}]\label{Lemma_signals}
	The randomly switching signal $r(t)$ has the following properties.
	\begin{enumerate}			
		\item[(1)]
		If $r(t)$ is an irreducible semi-Markov chain, that is its embedded Markov chain is irreducible, then it has stationary distribution \vspace{-0.2cm}
		\begin{equation*} 
			\lim_{t\rightarrow\infty}\mathbb{P}[r(t)=i]=\lim_{t\rightarrow\infty}\frac{T_i(t,t_0)}{t-t_0}=\pi_i,\ \mathrm{a.s.}
		\end{equation*}
		where
		\begin{equation*} \vspace{-0.2cm}
			\pi_i=\frac{\bar{\pi}_i m_i}{\sum_{j\in\mathbb{S}}\bar{\pi}_j m_j},\ i\in\mathbb{S},
		\end{equation*}
		and $\bar{\pi}=[\bar{\pi}_i]_{1\times M}$ represents the stationary distribution of discrete embedded Markov chain. By ergodic theory, there hold \vspace{-0.2cm}
			\begin{equation*}
				\lim_{t\rightarrow\infty}\frac{N_i(t,t_0)}{t-t_0}=\frac{\pi_i}{m_i},\ \mathrm{a.s.}
			\end{equation*}
			and \vspace{-0.2cm}
			\begin{equation*}
				\lim_{t\rightarrow\infty}\frac{N_{ij}(t,t_0)}{t-t_0}=\frac{\pi_ip_{ij}}{m_{i}},\ \mathrm{a.s.}
			\end{equation*}
			where $\mathbb{E}[{S_i(k)}]=m_i,\ \forall k\in\mathbb{N}^+$.
		\item[(2)] 
		If $r(t)$ is an irreducible Markov chain has stationary \vspace{-0.2cm}
		\begin{equation*}
			\lim_{t\rightarrow\infty}\mathbb{P}[r(t)=i]=\pi_i,\ \mathrm{a.s.}
		\end{equation*}
		and by its ergodicity, there hold \vspace{-0.2cm}
			\begin{equation*}
				\lim_{t\rightarrow\infty}\frac{N_i(t,t_0)}{t-t_0}=\pi_i q_i,\ \mathrm{a.s.}
			\end{equation*}
			and \vspace{-0.2cm}
			\begin{equation*}
				\lim_{t\rightarrow\infty}\frac{N_{ij}(t,t_0)}{t-t_0}=\pi_i p_{ij} q_i=\pi_i q_{ij},\ \mathrm{a.s.}
			\end{equation*}
			where $q_i=|q_{ii}|,\ i,j\in\mathbb{S}$, $i\neq j$.
		\item[(3)]
		If $r(t)$ is a renewal process, which means $N(t,t_0)$ is a renewal process, then \vspace{-0.2cm}
		\begin{equation*}
			\lim_{t\rightarrow\infty}\frac{N(t,t_0)}{t-t_0}=\frac{1}{\theta},\ \mathrm{a.s.}
		\end{equation*}
		where $\mathbb{E}[S(k)]=\theta$.
		And there exists a probability distribution \vspace{-0.2cm}
		\begin{equation*}
			\mathbb{P}[r(t_{k+1})=i|r(t_k),\cdots,r(t_0)]=p_i.
		\end{equation*}
	\end{enumerate}
\end{lemma}

\begin{definition}\label{Def:Stabilities}
	System \eqref{eq:SDEwMS} is said to be:
	\begin{enumerate}
		\item[(1)]
		almost surely uniformly stable (US a.s.), if for each $\epsilon>0$, there exists a constant $\delta=\delta(\epsilon)>0$, such that when $|\phi|<\delta$, \vspace{-0.2cm}
		\begin{equation*}
			\mathbb{P}\left[\sup_{t\geq t_0}|x(t)|<\epsilon\right]=1;
		\end{equation*} 
		\item[(2)]
		almost surely uniformly attractive (UA a.s.), if for any $t_0,\epsilon'>0$, there exist $\rho>0$ and $T=T(\epsilon')\geq0$ so that when $|\phi|<\delta$, \vspace{-0.2cm}
		\begin{equation*}
			\mathbb{P}\left[\sup_{t\geq T+t_0}|x(t)|<\epsilon'\right]=1;
		\end{equation*}
		\item[(3)]
		almost surely globally uniformly asymptotically stable (GUAS a.s.), if (1) and (2) are both fulfilled;
		\item[(4)]
		almost surely globally exponential stable (GES a.s.), if \vspace{-0.2cm}
		\begin{equation}\label{ineq:GES}
			\limsup_{t\rightarrow\infty}\frac{1}{t}\ln|x(t)|<0,\ \mathrm{a.s.}
		\end{equation}
		holds for all $\phi\in\mathbb{R}^n$.
	\end{enumerate}
\end{definition}

\vspace{-0.5cm}
\section{Main Results}
In this section, we will present a novel approach called mean uniformly functions (MUSFs) condition for iMLFs to establish sufficient conditions for almost sure GUAS and GES. An illustrative example will be proposed to demonstrate the reasonableness of MUSFs and its effectiveness in the stability analysis of randomly switching time-varying systems.

\vspace{-0.5cm}
\subsection{An illustrative example}\label{demo_example}

In this subsection, we will present a example to illustrate the idea of presenting mean uniformly stable functions.

	Consider a scalar linear time-varying system \vspace{-0.2cm}
	\begin{equation*}
		\begin{cases}
			\dot{x}(t)=a_{r}(t)x(t),\\
			x(t_0)=x(0)=x_0,
		\end{cases}
	\end{equation*}
	where $r(t)\in\mathbb{S}$ is an irreducible semi-Markov chain assumed to be ergodic with a unique stationary distribution. Its solution is given by \vspace{-0.4cm}
	\begin{equation*}
		x(t)=x_0\exp\left\{\int_{0}^{t}a_{r(s)}(s)\mathrm{d}s\right\}
	\end{equation*}
	To verify stability, we introduce the Lyapunov exponent \cite{Feng1992stochastic}: \vspace{-0.2cm}
	\begin{align}
		&\lim_{t\rightarrow\infty}\frac{1}{t}\ln|x(t)|\notag\\
		=&\lim_{t\rightarrow\infty}\frac{1}{t}\sum_{i\in\mathbb{S}}\int_{T_i(t,0)}a_{i}(s)\mathrm{d}s\notag\\
		=&\lim_{t\rightarrow\infty}\sum_{i\in\mathbb{S}}\frac{N_i(t,0)}{t}\frac{1}{N_i(t,0)}\sum_{k=0}^{N_i(t,0)}\int_{t_{i_{k}}}^{t_{i_{k+1}}}a_{i}(s)\mathrm{d}s, \label{Lyapunov_exp}
	\end{align}
	where $t_{i_{k}}$ is the $k$th visiting of mode $i$ and $S_i(k)=t_{i_{k+1}}\!-\!t_{i_{k}}$ is the sojourn time of mode $i$ at its $k$th visiting.

	Let's consider that all the functions $a_i(s)$ are constants, that is $a_i(s)=\bar{a}_i\in\mathbb{R}$ for all $i\in\mathbb{S}$, then $\int_{t_{i_{k}}}^{t_{i_{k+1}}}a_{i}(s)\mathrm{d}s=\bar{a}_i(t_{i_{k+1}}-t_{i_k})=\bar{a}_i S_i(k)$. By the strong law of large number \cite{Durrett2019probability} and ergodicity of $r$, we have \vspace{-0.2cm}
	\begin{align}\label{aS_egodic}
		&\lim_{t\rightarrow\infty}\sum_{i\in\mathbb{S}}\frac{N_i(t,0)}{t}\frac{1}{N_i(t,0)}\sum_{k=0}^{N_i(t,0)}a_iS_i(k)\notag\\
		=&\sum_{i\in\mathbb{S}}\frac{\pi_i}{m_i}a_i m_i
		=\sum_{i\in\mathbb{S}}\pi_i a_i,\ \mathrm{a.s.}
	\end{align}
	If $\sum_{i\in\mathbb{S}}\pi_i a_i<0$, system will be GES a.s. In the above example, irreducibility of the semi-Markov process does play a critical in deriving almost sure stability. Meanwhile, we can notice that for each mode $i$ sequence $\left\{\bar{a}_i S_i(k)\right\}_{k\in\mathbb{N}^+}$ is an independent, identically distributed (i.i.d.) process and also a wide-sense stationary process. This fact also plays a crucial role in deducing \eqref{aS_egodic}, as we will discuss next.
	 	
	For each $i$, we define a random variable for sequence $\{\bar{a}_iS_i(k)\}$ which is its time average: 
	\begin{small}
		\begin{equation*}
			A_S^i=\lim_{N\rightarrow\infty}\frac{1}{N_i}\sum_{k=0}^{N_i}a_iS_i(k),
		\end{equation*}
	\end{small}
	and its expectation is
	\begin{small}
		\begin{equation*}
			\bar{A}_S^i=\mathbb{E}[A_S^i]=\lim_{N\rightarrow\infty}\frac{1}{N_i}\sum_{k=0}^{N_i}a_i\mathbb{E}[S_i(k)]=a_i m_i.
		\end{equation*}
	\end{small} 
	We then use Chebychev's inequality to verify \eqref{aS_egodic} holds: 
	\begin{small}
		\begin{equation}\label{Pr_A_S}
			\mathbb{P}[|A_S^i-\bar{A}_S^i|<\varepsilon]\geq1-\frac{\mathrm{Var}(A_S^i)}{\varepsilon^2},
		\end{equation}
	\end{small}
	where $\varepsilon$ is a positive constant. For our random process to be ergodic, we need the probability in \eqref{Pr_A_S} to be 1 no matter the value of $\varepsilon$ is. Now, we calculate the variance:
	\begin{small}
			\begin{align}\label{var_A}
			&\mathrm{Var}(A_S^i)\notag\\
			=&\mathbb{E}\left[(A_S^i-\bar{A}_S^i)^2\right]\notag\\
			=&\mathbb{E}\left[\lim_{N_i\rightarrow\infty}(\frac{1}{N_i}\sum_{k=0}^{N_i}\bar{a}_iS_i(k)-\bar{a}_im_i)(\frac{1}{N_i}\sum_{l=0}^{N_i}\bar{a}_iS_i(l)-\bar{a}_im_i)\right]\notag\\
			=&\lim_{N_i\rightarrow\infty}\left(\frac{1}{N_i}\right)^2\sum_{k=0}^{N_i}\sum_{l=0}^{N_i}\mathbb{E}\left[(\bar{a}_iS_i(k)-\bar{a}_im_i)(\bar{a}_iS_i(l)-\bar{a}_im_i)\right]\notag\\
			=&\lim_{N_i\rightarrow\infty}\left(\frac{1}{N_i}\right)^2\sum_{k=0}^{N_i}\sum_{l=0}^{N_i}\left(\mathbb{E}\left[\bar{a}_i^2 S_i(k) S_i(l)\right]-(\bar{a}_im_i)^2\right)\notag\\
			=&\lim_{N_i\rightarrow\infty}\left(\frac{1}{N_i}\right)^2\sum_{k=0}^{N_i}\sum_{l=0}^{N_i}\left(\bar{a}_i^2\mathbb{E}\left[S_i(k)\right]\mathbb{E}\left[S_i(l)\right]-(\bar{a}_im_i)^2\right)\notag\\
			=&0,
		\end{align}
	\end{small}
	where the 5th equal sign holds by the independence between $S(k)$ and $S(l)$, and the 6th equal sign holds by $\mathbb{E}[S(k)]=\mathbb{E}[S(l)]=m_i$.

	For each time-varying parameter $a_i(s)$, it's too strict to require sequence $\{\int_{t_{i_k}}^{t_{i_{k+1}}}a_i(s)\mathrm{d}s\}$ to be a wide-sense stationary and an i.i.d. process simultaneously. Therefore, we introduce functions $b_i\in\mathcal{C}(\mathbb{R}^+,\mathbb{R})$ such that  $\int_{t_{i_k}}^{t_{i_{k+1}}}a_i(s)\mathrm{d}s\leq b_i(S_i(k))$ holds for each $i$. Then, \eqref{Lyapunov_exp} becomes\vspace{-0.2cm}
	\begin{align}
		&\lim_{t\rightarrow\infty}\frac{1}{t}\ln|x(t)|\notag\\
		=&\lim_{t\rightarrow\infty}\sum_{i\in\mathbb{S}}\frac{N_i(t,0)}{t}\frac{1}{N_i(t,0)}\sum_{k=0}^{N_i(t,0)}\int_{t_{i_{k}}}^{t_{i_{k+1}}}a_{i}(s)\mathrm{d}s,\notag\\
		\leq&\lim_{t\rightarrow\infty}\sum_{i\in\mathbb{S}}\frac{N_i(t,0)}{t}\frac{1}{N_i(t,0)}\sum_{k=0}^{N_i(t,0)}b_i(S_i(k)).
	\end{align}
	For any mode $i$, function $b_i(s)$ only depends on its corresponding sojourn time interval $S_i(k)$. Define $Y^i_k=b_i(S_i(k))$ for each $i$, and thus the sequence $\{Y^i_k\}_{k\in\mathbb{N}^+}$ is an i.i.d. process. Moreover, the continuity inherit from integral $\int_{t_{i_k}}^{t_{i_{k+1}}}a_i(s)\mathrm{d}s$ makes each $b_i(s)$ to be a measure-preserving function. Hence, for all $k$, $\mathbb{E}[Y^i_{k+1}]=\mathbb{E}[Y^i_k]=M_i$, $\{Y^i_k\}_{k\in\mathbb{N}^+}$ is a wide-sense stationary process.

Same as proving \eqref{aS_egodic}, we can define a random variable $A_y^i$ from $\{Y_k^i\}_{k\in\mathbb{N}^+}$ as $A_y^i=\lim_{N_i\rightarrow\infty}\frac{1}{N_i}\sum_{k=0}^{N_i}Y_k^i$ and $\bar{A}_y^i=\mathbb{E}[Y_k^i]=M_i$, where $N_i=N_i(t,t_0)$, $N_i\rightarrow\infty$ as $t\rightarrow\infty$ and prove that the sequence $\{Y^i_k\}_{k\in\mathbb{N}^+}$ is ergodic by the Chebychev's inequality. The Lyapunov exponent \eqref{Lyapunov_exp} follows from the ergodicity of $\{Y^i_k\}_{k\in\mathbb{N}^+}$ that \vspace{-0.2cm}
	\begin{align*}
		\lim_{t\rightarrow\infty}\frac{1}{t}\ln|x(t)|
		\leq&\lim_{t\rightarrow\infty}\sum_{i\in\mathbb{S}}\frac{N_i(t,0)}{t}\frac{1}{N_i(t,0)}\sum_{k=0}^{N_i(t,0)}b_i(S_i(k))\\
		=&\sum_{i\in\mathbb{S}}\frac{\pi_i}{m_i}M_i,\ \mathrm{a.s.}
	\end{align*}
When $\sum_{i\in\mathbb{S}}\frac{\pi_i}{m_i}M_i<0$, system is GES a.s. We derive a sufficient conditions of GES a.s. for time-varying systems.   

\begin{remark}
	In this example, the ergodicity holds due to the wide-sense stationarity and independence of the processes $\{Y_k^i\}_k$. For each subsystem, the wide-sense stationarity of $\{Y_k^i\}_k$ implies that $\mathbb{E}\left[\int_{t_{i_k}}^{t_{i_{k+1}}}a_i(s)\mathrm{d}s\right]\leq M_i$ for all $k$, indicating that each subsystem is mean bounded during its corresponding sojourn time intervals. This example illustrates that the prerequisite for an almost sure stable system is that all subsystems must be uniformly mean bounded.
\end{remark}

\vspace{-0.4cm}
\subsection{Indefinite multiple Lyapunov functions and mean uniformly functions}

The iMLFs scheme is given as follow.

\begin{assumption}\label{iMLFs}
	For all $i,j\in\mathbb{S}$,  $V_i\in\mathcal{C}^{1,2}(\mathbb{R}_{t_0}^+,\mathbb{R}^n;\mathbb{R}^+)$ are said to be iMLFs, if there exists $\lambda_i\in\mathcal{PC}(\mathbb{R}_{t_0}^+;\mathbb{R})$, $\alpha_1,\alpha_2\in\mathcal{K}_{\infty}$, $\rho\in\mathcal{K}_{\infty}$, and positive constants $\mu_{ij}>0$ such that
	\begin{enumerate}
		\item[(A.1)]
		$\alpha_1(|x(t)|)\leq V_i(t,x(t))\leq\alpha_2(|x(t)|),\ \forall i\in\mathbb{S};$
		\item[(A.2)]
			for all $t\in\mathbb{R}_{t_0}^+$ and $i\in\mathbb{S}$, \vspace{-0.3cm}
			\begin{equation}\label{V_sep_ISS}
				\dot{V}_i(t,x)\leq\lambda_i(t)V_i(t,x(t))
			\end{equation}
			where $\dot{V}_i(t,x)=\frac{\partial V_i(t,x)}{\partial t}+\frac{\partial V_i(t,x)}{\partial x}f_i(t,x)$;
		\item[(A.3)]
			for all $t\in\mathcal{T}$ and $i,j\in\mathbb{S}$, $V_i(t,x(t))\leq\mu_{ij} V_j(t,x(t))$ holds with $\mu_{ii}=1$.
	\end{enumerate}
\end{assumption}

\begin{remark}
	Assumption (A.1) implies that these Lyapunov-like functions are radially unbounded. Assumption (A.2) gives time-varying stability estimators for all subsystems with time-varying dynamic. The indefinite sign of $\lambda(t)$ in \eqref{V_sep_ISS} implies that subsystems may be stable or unstable during the whole timeline. The switching compatibility is assumed in (A.3) among the Lyapunov-like functions, which effectively establishes mode-dependent linear comparable relationships between different Lyapunov-like functions.
\end{remark}

Inspired by the discussions in the Subsection 3.1, to make iMLFs method applicable for randomly switched time-varying systems, we introduce following condition on each $\lambda_{i}(t)$, named as mean uniformly stable function.

\begin{definition}\label{Def_MUSF}
	Time-varying scalar function $\lambda(t)\in\mathcal{PC}(\mathbb{R}^+;\mathbb{R})$ is said to be a mean uniformly stable function (MUSF), if for corresponding real-valued random variable $S>0$ with $\mathbb{E}[S]<\infty$, there exists a continuous function $\varphi(S)\in\mathcal{C}(\mathbb{R}^+;\mathbb{R})$ such that $\mathbb{E}\left[\varphi(S)\right]<\infty$ and
	\begin{equation}\label{MUSF}
		\int_{t}^{t+S}\lambda(h)\mathrm{d}h	\leq \varphi(S),\ \forall t\geq t_0.
	\end{equation}
\end{definition}

\begin{remark}
		When applying the MUSF condition within the iMLF framework, we select the sojourn time $S_i$ as the corresponding random variable for $\lambda_i(t)$. This allows us to compute the integral $\int_{t}^{t+S_i} \lambda_i(h)\mathrm{d}h$ and derive the corresponding function $\varphi_i(\cdot)$. In this process, we treat $S_i$ as a coefficient in the definite integral calculation, yielding $\varphi_i(S_i)$. Then, by taking the expectation of both $\int_{t}^{t+S_i}\lambda_i(h)\mathrm{d}h$ and $\varphi_i(S_i)$, we obtain a real-valued bound that captures the cumulative variations of the time-varying function $\lambda_i(t)$ over the random sojourn time $S_i$. Due to the piecewise continuity of function $\lambda_i(t)$, its integral is continuous, and existence of a continuous function $\varphi_i(\cdot)$ bounding this integral is reasonable. Meanwhile, function $\varphi_i$ is independent on time but dependent on sojourn time intervals, which performs the uniformity in time. $\mathbb{E}[\int_{t}^{t+S_i}\lambda_i(h)\mathrm{d}h]\leq\mathbb{E}[\varphi_i(S_i)]$ implies that all subsystems are mean bounded during their respective sojourn time intervals. Thus, we name this constraint for time-varying scalar function $\lambda_i(t)$ as `mean uniformly stable function'.
\end{remark}

\vspace{-0.4cm}
\subsection{Almost sure stability of time-varying systems with semi-Markov switching}

\begin{theorem}\label{SM_GUAS}
	If Assumption \ref{iMLFs} and following condition hold:
	\item[(S.4)] $\lambda_{i}(t)$ are MUSFs for all $i\in\mathbb{S}$ and $\sum_{i\in\mathbb{S}}\frac{\pi_i}{m_i}(\mathbb{E}(\varphi_i(S_i))+\sum_{j\in\mathbb{S}}p_{ij}\ln\mu_{ij})<0$.
	
	\noindent Then semi-Markov switched system \eqref{eq:SDEwMS} is GUAS a.s.
\end{theorem}

\begin{proof}
		Without loss of generality, we assume that $t_0=0$. To simplify, $v_{r}(t)=V_{r}(t,x(t)),r(t)\in\mathbb{S}$.
		By (A.3), we get that for $t\in[t_k,t_{k+1})$, $r(t)=r_k$, \vspace{-0.2cm}
		\begin{equation*}
			v_{r_k}(t)\leq v_{r_k}(t_k)e^{\int_{t_k}^{t}\lambda_{r_k}(h)\mathrm{d}h}.
		\end{equation*}
		and for $t\in[t_{l-1},t_l]$, $r(t)=r_{l-1}$, $l=1,2,\cdots,k$, \vspace{-0.2cm}
		\begin{equation*}
			v_{r_l}(t_{l+1}^-)=v_{r_l}(t_l)e^{\int_{t_l}^{t_{l+1}}\lambda_{r_l}(h)\mathrm{d}h}.
		\end{equation*}
		Then (A.4) yields that  $v_{r_{l+1}}(t_{l+1})\leq\mu_{r_{l+1}r_{l}}v_{r_l}(t_{l+1}^-).$
		Thus, by iteration, there holds \vspace{-0.2cm}
		\begin{equation}\label{v_est}
			v_{r}(t)\leq v_{r_0}(0)\prod_{i,j\in\mathbb{S}}\mu_{ij}^{N_{ij}(t,0)}e^{\int_{0}^{t}\lambda_{r(h)}(h)\mathrm{d}h}.
	    \end{equation}
	Condition (A.1) yields that
	\begin{equation}\label{x_est}
			|x(t)|\!\leq\!\alpha_1^{\!-\!1}\!\left(\!\alpha_2(|\phi|) e^{\int_{0}^t\lambda_{r(h)}(h)\mathrm{d}h\!+\!\sum_{i,j\in\mathbb{S}}N_{ij}(t,0)\!\ln\mu_{ij}}\!\right).
	\end{equation}
	We claim that: 
		\begin{equation}\label{lambda_limt}
			\lim_{t\rightarrow\infty}\frac{1}{t}\int_{0}^{t}\lambda_{r(h)}(h)\mathrm{d}h\leq\sum_{i\in\mathbb{S}}\frac{\pi_i}{m_i}\mathbb{E}[\varphi_{i}(S_i)],\ \mathrm{a.s.}
	\end{equation}
	The proof of \eqref{lambda_limt} is given as following: \vspace{-0.2cm}
	\begin{small}
	\begin{align}\label{lambda_limt_1}
		&\lim_{t\rightarrow\infty}\frac{1}{t}\int_{0}^{t}\lambda_{r(h)}(h)\mathrm{d}h\notag\\
		=&\lim_{t\rightarrow\infty}\frac{1}{t}\sum_{i\in\mathbb{S}}\sum_{k=0}^{N_i(t,0)}\int_{t_{i_{k}}}^{t_{i_{k+1}}}\lambda_{i}(h)\mathrm{d}h\notag\\
		=&\lim_{t\rightarrow\infty}\sum_{i\in\mathbb{S}}\frac{N_i(t,0)}{t}\frac{1}{N_i(t,0)}\sum_{k=0}^{N_i(t,0)}\int_{t_{i_{k}}}^{t_{i_{k+1}}}\lambda_{i}(h)\mathrm{d}h\notag\\
		\leq&\lim_{t\rightarrow\infty}\sum_{i\in\mathbb{S}}\frac{N_i(t,0)}{t}\frac{1}{N_i(t,0)}\sum_{k=0}^{N_i(t,0)}\varphi_{i}(S_i(k)).
	\end{align}
	\end{small}
	Since $\lambda_{i}(t)$ are MUSFs, we have for $k=1,\!\cdots\!,N_i(t,0)$, $\int_{t_{i_{k}}}^{t_{i_{k+1}}}\!\lambda_{i}(h)\mathrm{d}h\!\leq\!\varphi_{i}(S_i(k))$. For each $i$, sequence $\{\varphi_{i}(S_i(k))\}$ is a wide-sense stationary process. Then, we repeat the proof procedure of the illustrative example in Section 3.1 and get \vspace{-0.2cm}
		\begin{small}
		\begin{equation*}
			\lim_{t\rightarrow\infty}\frac{1}{N_i(t,0)}\sum_{k=0}^{N_i(t,0)}\varphi_{i}(S_i(k))
			=\mathbb{E}\left[\varphi_{i}(S_i(0))\right],\ \mathrm{a.s.}
		\end{equation*}
		\end{small}
		Recalling the Lemma \ref{Lemma_signals} and applying strong law of large numbers \cite{Durrett2019probability}, following holds almost surely \vspace{-0.2cm}
		\begin{equation*}
			\eqref{lambda_limt_1}
			\!=\!\lim_{t\rightarrow\infty}\sum_{i\in\mathbb{S}}\frac{N_i(t,0)}{t}\mathbb{E}\left[\varphi_{i}(S_i(0))\right]\!=\!\sum_{i\in\mathbb{S}}\frac{\pi_i}{m_i}\mathbb{E}\left[\varphi_{i}(S_i(0))\right],
		\end{equation*}
		which yields \eqref{lambda_limt}.
	
		It follows from \eqref{lambda_limt}, (S.4) and Lemma \ref{Lemma_signals} that \vspace{-0.2cm}
		\begin{small}
		\begin{align}\label{limt_lambda_Nt}
			&\lim_{t\rightarrow\infty}\frac{1}{t}\left(\int_{0}^{t}\lambda_{r(h)}(h)\mathrm{d}h+\sum_{i,j\in\mathbb{S}}N_{ij}(t,0)\ln\mu_{ij}\right)\notag\\
			\leq&\sum_{i\in\mathbb{S}}\frac{\pi_i}{m_i}\left(\mathbb{E}\left[\varphi_{i}(S_i(0))\right]+\sum_{j\in\mathbb{S}}p_{ij}\ln\mu_{ij}\right)<0,\ \mathrm{a.s.}
		\end{align}
		\end{small}	
		which implies that \vspace{-0.2cm}
		\begin{small}
		\begin{equation}\label{lim_lambda<0}
			\lim_{t\rightarrow\infty}\!\left(\!\int_{0}^{t}\!\lambda_{r(h)}(h)\mathrm{d}h\!+\!\sum_{i,j\in\mathbb{S}}\!N_{ij}(t,0)\!\ln\!\mu_{ij}\!\right)\!=\!-\!\infty,\ \mathrm{a.s.}
		\end{equation}
		\end{small}
		and it can further deduce that \vspace{-0.2cm}
		\begin{equation*}
			\lim_{t\rightarrow\infty}e^{\int_{0}^{t}\lambda_{r(h)}(h)\mathrm{d}h+\sum_{i,j\in\mathbb{S}}N_{ij}(t,0)\ln\mu_{ij}}=0,\ \mathrm{a.s.}
		\end{equation*}
		Hence, we can get that there exists a constant $M>0$ satisfying \vspace{-0.2cm}
		\begin{small}
		\begin{equation*}
			\sup_{t\geq 0}\int_{0}^{t}\lambda_{r(h)}(h)\mathrm{d}h+\sum_{i,j\in\mathbb{S}}N_{ij}(t,0)\ln\mu_{ij}\leq\ln M,\ \mathrm{a.s.}
		\end{equation*}
		\end{small}	
	Thus, for any $\varepsilon>0$, set $0<\delta<\alpha_2^{-1}(\alpha_1(\varepsilon)/M)$, when $|\phi|<\delta$ and $\|u\|<\delta$, we get from \eqref{x_est} that \vspace{-0.2cm}
	\begin{equation*}
		\sup_{t\geq 0}|x(t)|<\varepsilon,\ \mathrm{a.s.},
	\end{equation*}
	which satisfies US condition in Definition \ref{Def:Stabilities}.
	Moreover, we can know that: for any given constants $\delta'>0$ and $\epsilon'>0$, there exists $T=T(\varepsilon')$ such that \vspace{-0.2cm}
	\begin{equation*}
			\mathbb{P}\left[\sup_{t\geq T}\prod_{i\in\mathbb{S}}\mu_{ij}^{N_{ij}(t,0)}e^{\int_{0}^{t}\lambda_{r(h)}(h)\mathrm{d}h}<\frac{\alpha_1(\varepsilon')}{\alpha_2(\delta')}\right]=1.
	\end{equation*}
	Then, together with \eqref{x_est}, for any $|\phi|<\delta'$, there holds \vspace{-0.2cm}
	\begin{equation*}
		\mathbb{P}\left[\sup_{t\geq T}|x(t)|<\epsilon'\right]=1,
	\end{equation*}
	which satisfies UA condition in Definition \ref{Def:Stabilities}. Time-varying system \eqref{eq:SDEwMS} with semi-Markov switching is GUAS a.s.
\end{proof}
\vspace{-0.4cm}

\begin{remark}
		In order to explore the tolerance of Theorem \ref{SM_GUAS} for unstable subsystems, we might assume that: there exist sets $\mathbb{S}_1$ and $\mathbb{S}_2$ such that $\mathbb{S}=\mathbb{S}_1\cup\mathbb{S}_2$ and $\mathbb{S}_1\cap\mathbb{S}_2=\emptyset$, for the modes belonging to $\mathbb{S}_1$ relative subsystems are stable, and for the modes belonging to $\mathbb{S}_2$ subsystems are unstable and $\lambda_{j}(t)>0,\ \forall j\in\mathbb{S}_2$.
		Although for all $j\in\mathbb{S}_1$ $\lambda_{j}(t)$ are allowed to be positive to describe the existence of unstable subsystems, the MUSF condition yields integrals of $\lambda_{j}(t)$ are bounded $\mathbb{E}\left[\int_{t_{j_k}}^{t_{j_{k+1}}}\lambda_{j}(s)\mathrm{d}s\right]\leq\mathbb{E}[\varphi_{j}(S_j)]<\infty$ which implies $j$th subsystem is stable in mean. On the other hand, condition (S.4) implies that, for all $i\in\mathbb{S}_1$, $\mathbb{E}\left[\int_{t_{i_k}}^{t_{i_{k+1}}}\lambda_{i}(s)\mathrm{d}s\right]\leq\mathbb{E}[\varphi_{i}(S_i)]<0$ and $\mathbb{E}[\varphi_{i}(S_i)]$ need to be negative enough to ensure the whole system stability.
\end{remark}

\vspace{-0.4cm}
\begin{theorem}\label{SM_GES}
	If (A.2), (A.3), (S.4) and following hold
	\begin{enumerate}
		\item [(B.1)] there exists positive constants $c,p>0$ such that
		\begin{equation*}
			c|x(t)|^p\leq V_i(t,x(t)),\ \forall i\in\mathbb{S}.
		\end{equation*}
	\end{enumerate}
	Then time-varying system \eqref{eq:SDEwMS} with semi-Markov switching signal is GES a.s.
\end{theorem}

\begin{proof}
	Recalling \eqref{v_est}, condition (B.2) and (A.3) yield that
	\begin{equation*}
		V_r(t,x(t))=V_{r_0}(t_0,\phi)\prod_{i\in\mathbb{S}}\mu_{ij}^{N_{ij}(t,t_0)}e^{\int_{t_0}^{t}\lambda_{r(h)}(h)\mathrm{d}h}
	\end{equation*}
	Taking logarithm on both sides, we have
	\begin{align}\label{logv_est}
		\ln V_r(t,x)\leq&\ln V_{r_0}(t_0,\phi)+\sum_{i,j\in\mathbb{S}}N_{ij}(t,t_0)\ln\mu_{ij}\notag\\
		&+\int_{t_0}^{t}\lambda_{r(h)}(h)\mathrm{d}h
	\end{align}
	It can be checked by (B.1) and \eqref{lambda_limt} that
		\begin{small}
		\begin{align*}
			&\limsup_{t\rightarrow\infty}\frac{1}{t}\ln|x(t)|\\
			\leq&\frac{1}{p}\limsup_{t\rightarrow\infty}\frac{1}{t}\left(\ln V_{r_0}(t_0,\phi)\right.\\
			&\left.+\sum_{i,j\in\mathbb{S}}N_{ij}(t,t_0)\ln\mu_{ij}+\int_{t_0}^{t}\lambda_{r(h)}(h)\mathrm{d}h+\ln c\right)\\
			=&\frac{1}{p}\sum_{i\in\mathbb{S}}\frac{\pi_i}{m_i}\left(\mathbb{E}\left[\varphi_{i}(S_i(0))\right]+\sum_{j\in\mathbb{S}}p_{ij}\ln\mu_{ij}\right),\ \mathrm{a.s.}
		\end{align*}
		\end{small}	
	Then, together with (S.4), we get \eqref{ineq:GES}. Time-varying system with semi-Markov switching is GES a.s.
\end{proof}

\subsection{Almost sure stability of time-varying systems with Markov switching}

A Markov chain is a special form of a semi-Markov chain, and it can be directly derived  from a semi-Markov chain by assuming that sojourn time satisfies exponential distributions (see Lemma 2 in \cite{Wu2018}). Therefore, we can promptly obtain stability criteria for the system \eqref{eq:SDEwMS} with a Markov switching signal by letting $m_i={1}/{q_i}$.

\begin{corollary}\label{M_GUAS}
	If Assumption \ref{iMLFs} and following condition hold:
	\item[(M.4)]for all $i\in\mathbb{S}$, $\lambda_{i}(t)$ are MUSFs 
	and $\sum_{i\in\mathbb{S}}(\mathbb{E}[\varphi_i(S_i)]q_i+\sum_{j\in\mathbb{S}}q_{ij}\ln\mu_{ij})\pi_i<0$.

	\noindent Then system \eqref{eq:SDEwMS} with Markov switching is GUAS a.s.
\end{corollary}

\begin{corollary}\label{MS_GES}
	If (B.1), (A.2), (A.3) hold and (M.4) is satisfied.	
	Then system \eqref{eq:SDEwMS} with Markov switching is GES a.s.
\end{corollary}
\vspace{-0.2cm}

\subsection{Almost sure stability of time-varying systems with renewal process switching}

\begin{theorem}\label{renewal_GUAS}
	If Assumption \ref{iMLFs} is satisfied with $\mu_{ij}=\mu>1$ and following condition holds:
	\item[(R.4)] for all $i\in\mathbb{S}$, $\lambda_{i}(t)$ are MUSFs $\int_{t}^{t+S}\lambda_{i}(h)\mathrm{d}h\leq\varphi_i(S)$,
		and $\frac{1}{\theta}(\sum_{i\in\mathbb{S}}\mathbb{E}[\varphi_i(S)] p_i+\ln\mu)<0$.
	
	\noindent Then system \eqref{eq:SDEwMS} with renewal process switching is GUAS a.s.
\end{theorem}

\begin{theorem}\label{renewal_GES}
	If (B.1), (A.2), (A.3) are satisfied with $\mu_{ij}=\mu>1$ and conditions (R.4) holds.
	Then system \eqref{eq:SDEwMS} with renewal process switching is GES a.s.
\end{theorem}

\begin{proof}
	The proof procedure of Theorems \ref{renewal_GUAS} and \ref{renewal_GES} is very similar to Theorems \ref{SM_GUAS} and \ref{SM_GES}, we only need to prove:\vspace{-0.2cm}
	\begin{equation}\label{ineq:Renewal_lambda}
			\lim_{t\rightarrow\infty}\frac{1}{t}\int_{t_{0}}^{t}\lambda_{r(h)}(h)\mathrm{d}h=\frac{1}{\theta}\sum_{i\in\mathbb{S}}\mathbb{E}\left[\varphi_i(S(0))\right]p_i,\ \mathrm{a.s.}
	\end{equation}

	It follows from the MUSFs $\lambda_{i}(t)$ that
		\begin{small}
			\begin{equation*}
				\int_{t_{k}}^{t_{k\!+\!1}}\!\lambda_{r(t_{k})}(h)\mathrm{d}h\!=\!\int_{t_{k}}^{t_{k\!+\!1}}\!\lambda_{i}(h)I_{r(t_{k})=i}\mathrm{d}h\\
				\!\leq\!\varphi_i(S(k))I_{r(t_{k})=i}.
			\end{equation*}
		\end{small}
	Since the sequence $\{S(k)\}_{k\in\mathbb{N}^+}$ is an i.i.d. process, $\left\{ \varphi_i(S(k)) \right\}_{k\in\mathbb{N}^+}$ is also an i.i.d. process. $\mathbb{E}[\varphi_i(S(k))]=\mathbb{E}[\varphi_i(S(k+1))]$ holds for all $k$ naturally.
	Define $Z_k:=\{\varphi_i(S(k))I_{r(t_{k})=i}\}$, we can notice that $Z_k$ is an i.i.d. process and
		\begin{small}
		\begin{align*}
			\mathbb{E}[Z_k]=&\mathbb{E}\left[\varphi_i(S(k))I_{r(t_k)=i}\right]=\sum_{i\in\mathbb{S}}\mathbb{E}\left[\varphi_i(S(k))\right]\mathbb{P}[r(t_k)=i]\\
			=&\sum_{i\in\mathbb{S}}\mathbb{E}\left[\varphi_i(S(k))\right]p_i=\sum_{i\in\mathbb{S}}\mathbb{E}\left[\varphi_i(S(k+1))\right]p_i\\
			=&\sum_{i\in\mathbb{S}}\mathbb{E}\left[\varphi_i(S(k+1))\right]\mathbb{P}[r(t_{k+1})=i]
			=\mathbb{E}[Z_{k+1}].
		\end{align*}
		\end{small}
	Thus, $\{Z_k\}, k=0,1,\cdots,N(t,t_0)$ is a wide-sense stationary process and also satisfies ergodicity. Then, for $l=0,1$, there holds \vspace{-0.2cm}
		\begin{small}
		\begin{align}\label{renew_ergodic}
			&\lim_{N(t,t_0)\rightarrow\infty}\frac{1}{N(t,t_0)+l}\sum_{k=0}^{N(t,t_0)+l}\int_{t_{k}}^{t_{k+1}}\lambda_{r(t_{k})}(h)\mathrm{d}h\notag\\
			\leq&\lim_{N(t,t_0)\rightarrow\infty}\frac{1}{N(t,t_0)+l}\sum_{k=0}^{N(t,t_0)+l}\varphi_{r(t_{k})}(S(k))\notag\\
			=&\mathbb{E}\left[\varphi_i(S(0))I_{r(t_0)=i}\right],\ \mathrm{a.s.}
		\end{align}
		\end{small}
	We can promptly prove \eqref{ineq:Renewal_lambda} by \eqref{renew_ergodic}. This completes the proof.

	\end{proof}
%
%
%

\vspace{-0.6cm}
\section{Numerical examples}

{\textbf{Example 1.}} Consider a time-varying system with a semi-Markov switching as follows:
\begin{equation}\label{example_1}
	\dot{x}(t)=f_{r}(t,x(t))
\end{equation}
where $f_1(t,x)=[-2x_1+x_2,x_1-2x_2]^T$, $f_2(t,x)=[-\frac{3}{2}t^2x_1,\cos t x_1-\frac{3}{2}t^2x_2]^T$, $f_3(t,x)=[\frac{1}{2}(t\cos t-\frac{1}{2}) x_1-x_1,x_1-\frac{1}{2}x_2+\frac{1}{2}t\left(\cos t-\frac{1}{2}\right)x_2]^T$.
The switching signal $r(t)$ is a semi-Markov chain with a embeded Markov chain determined by transition probability matrix 
\begin{small}
\begin{equation*}
	P=\left[
	\begin{array}{ccc}
		0&  0.8&  0.2\\
		0.7&  0&  0.3\\
		0.6&  0.4& 0
	\end{array}
	\right].
\end{equation*}
\end{small}
The unique stationary distribution of $r_k$ can be got $\bar{\pi}=\left[0.4,0.4,0.2\right]$. The sojourn time expectations of all the modes are $m_1=1$, $m_2=3$ and $m_3=2$. We thus get stationary distribution of $r(t)$ is $\pi=\left[0.2,0.6,0.2\right]$.

Choosing Lyapunov functions $V_1(t,x)\!=\!\frac{1}{2}\left(x_1^2\!+\!x_2^2\right),\ V_2(t,x)\!=\!x_1^2\!+\!x_2^2,\ V_3(t,x)\!=\!\frac{1}{2}x_1^2\!+\!x_2^2,$,
we thus get $\mu_{21}=\mu_{23}=2$, $\mu_{12}=\mu_{32}=0.5$ and $\mu_{13}=\mu_{31}=1$.
By calculating, we get that $\dot{V}_1(t,x)\leq-2V_1(t,x),\ \dot{V}_2(t,x)\leq(-3t^2+\cos t) V_3(t,x),\ \dot{V}_3(t,x)\leq t\left(\cos t-\frac{1}{2}\right)V_3(t,x),$
which implies that $\lambda_1(t)=-2,\ \lambda_2(t)=-3t^2+\cos t,\ \lambda_3(t)=t\left(\cos t-\frac{1}{2}\right).$
By applying Theorem \ref{SM_GUAS}, we know that system \eqref{example_1} is GUAS a.s. if all the $\lambda_{i}(t)$ are MUSFs and  $\sum_{i\in\mathbb{S}}\frac{\pi_i}{m_i}(\mathbb{E}(\varphi_i(S_i))+\sum_{j\in\mathbb{S}}p_{ij}\ln\mu_{ij})<0$.
Here, we verify \vspace{-0.2cm}
\begin{small}
	\begin{align*}
		&\mathbb{E}\left[\int_{t}^{t+S_1}\lambda_{1}(h)\mathrm{d}h\right]=\mathbb{E}[-2S_1]=-2m_1=-2,\\
		&\mathbb{E}\left[\int_{t}^{t+S_2}\lambda_{2}(h)\mathrm{d}h\right]\leq\mathbb{E}[-S_2^3+2]\leq-m_2^3+2=-7,\\
		&\mathbb{E}\left[\int_{t}^{t+S_3}\lambda_{3}(h)\mathrm{d}h\right]\leq\mathbb{E}[-\frac{1}{4}S_3^2+2S_3+2]\\
		&\quad\quad\quad\quad\quad\quad\quad\quad\quad\leq-\frac{1}{4}m_3^2+2m_3+2=5,
	\end{align*}
\end{small}
and $\sum_{i\in\mathbb{S}}\frac{\pi_i}{m_i}(\mathbb{E}(\varphi_i(S_i))+\sum_{j\in\mathbb{S}}p_{ij}\ln\mu_{ij})<-1.39$.
We can conclude that system \eqref{example_1} with semi-Markov switching signal is GUAS a.s.

Figure \ref{fig:examplefig1} shows the state trajectories of system \eqref{example_1} with semi-Markov switching under single experiment. The dash-dotted line represents the trajectory of $x_1$ and the dashed line corresponds to the trajectory of $x_2$. It is evident from the figure that the states converge to $0$ under a single experiment of semi-Markov switching. Figure \ref{fig:examplefigx1} and \ref{fig:examplefigx2} present trajectories under 2000 times semi-Markov switching experiments. In all trajectories, the states uniformly asymptotically converge to the origin, confirming the almost sure GUAS property.


\begin{figure}
	\centering
	\includegraphics[width=0.8\linewidth]{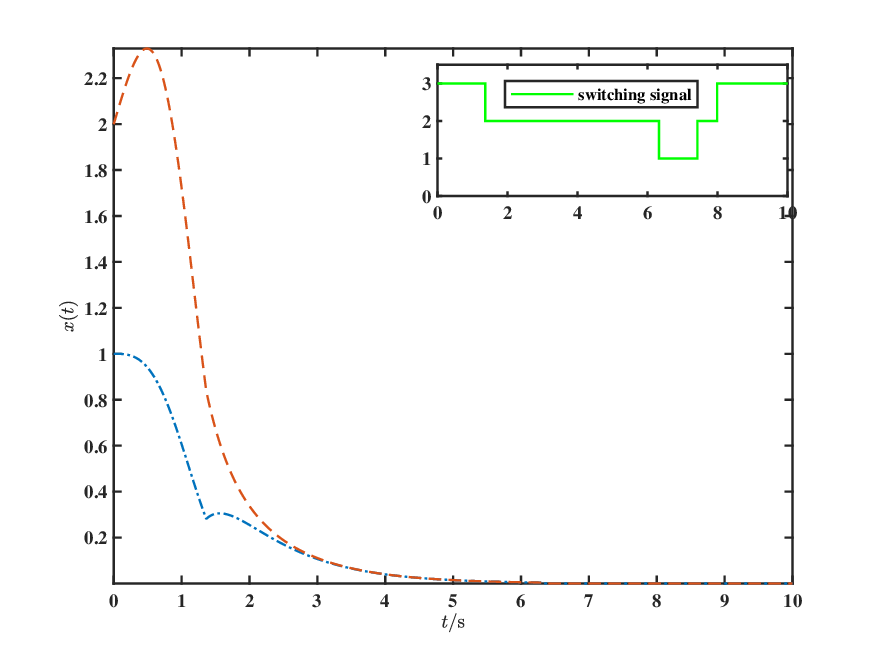}
	\caption{States' trajectories of system \eqref{example_1} under a single experiment.}
	\label{fig:examplefig1}
\end{figure}

\begin{figure}
	\centering
	\includegraphics[width=0.8\linewidth]{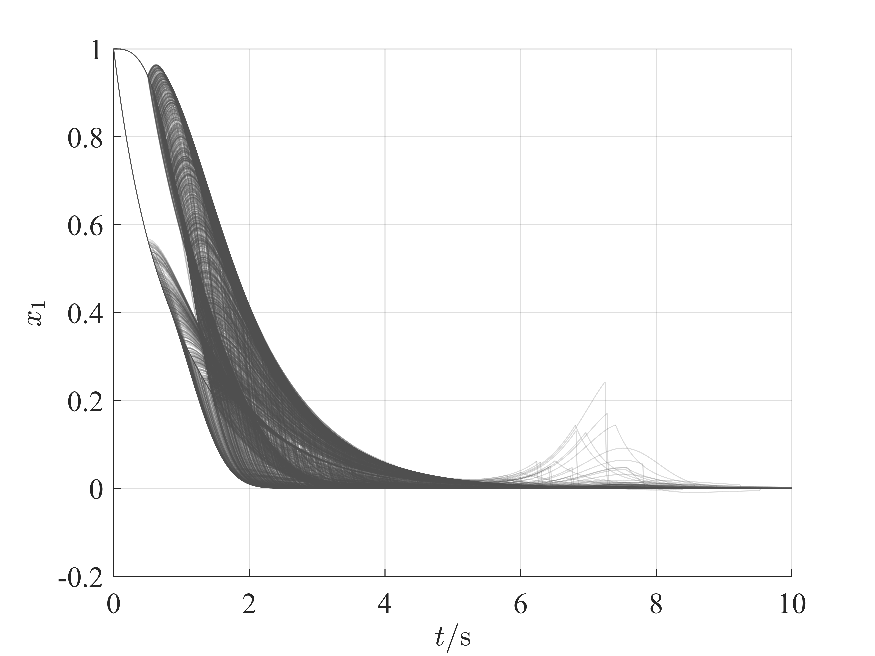}
	\caption{States' trajectories of system \eqref{example_1} under 2000 experiments.}
	\label{fig:examplefigx1}
\end{figure}

\begin{figure}
	\centering
	\includegraphics[width=0.8\linewidth]{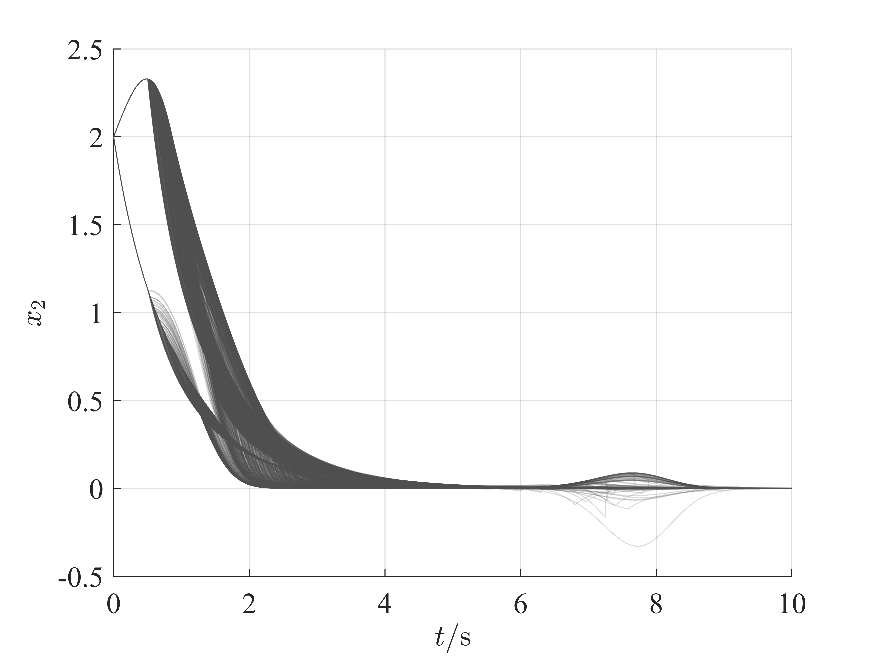}
	\caption{States' trajectories of system \eqref{example_1} under 2000 experiments.}
	\label{fig:examplefigx2}
\end{figure}

{\bf Example 2.} Let's consider the system \eqref{example_1} with following subsystems and a Markov switching signal:
$f_1=[\frac{1}{2}x_1,\frac{1}{2}x_2]^T$, $f_2=[-\frac{1}{2}(3t+1)x_1+x_2,\cos(x_1)x_2-(\frac{3t}{2}+1)x_2]^T$ and $f_3=[\frac{1}{2}(-2+\cos 4t) x_1,x_1+\frac{1}{2}(-3+\cos 4t) x_2]^T$. 
Switching signal $r(t)$ is a Markov chain with transition rate matrix 
\begin{small}
\begin{equation*}
	Q=\left[
	\begin{array}{ccc}
		-2&  1&  1\\
		1&  -2&  1\\
		2&  2&  -4
	\end{array}
	\right].
\end{equation*}
\end{small}
We get from $Q$ that stationary distribution is $\pi\!=\![0.4, 0.4, 0.2]$. 

We choose Lyapunov functions for subsystems: $V_1(t,x)=\frac{1}{2}\left(x_1^2+x_2^2\right),\ V_2(t,x)=x_1^2+x_2^2,\ V_3(t,x)=\frac{1}{2}x_1^2+x_2^2$, and thus get $\mu_{21}=\mu_{23}=2$, $\mu_{12}=\mu_{32}=0.5$ and $\mu_{13}=\mu_{31}=1$.
It can check that $\lambda_{1}(t)=1,\ \lambda_{2}(t)=-6t,\ \lambda_{3}(t)=-2+\cos 4t$ and the 1st subsystem is unstable, and the 2nd and 3rd subsystems are stable. Then, by applying the Corollary \ref{M_GUAS}, we can verify that $\sum_{i\in\mathbb{S}}\left(\mathbb{E}\left[\varphi_i(S_i)\right]q_i+\sum_{j\in\mathbb{S}}q_{ij}\ln\mu_{ij}\right){\pi_i}<-1.17.$
Hence, system \eqref{example_1} with Markov switching signal is GUAS a.s. 
This example shows that our results still work for randomly switched systems with unstable subsystems.

\section{Conclusion}
In this paper, we apply the iMLF method to analyze the stability of randomly switched time-varying systems. To achieve this, we present a novel MUSF condition for iMLFs. In establishing sufficient conditions for almost sure stability, we observe that, the mean boundedness of each subsystem over its sojourn time interval supports the almost stability of the entire system. This observation allow us to relax the stability criterion's requirements for individual subsystems, even allowing the existence of unstable subsystems. Additionally, we provide numerical examples to demonstrate the effectiveness and advantages of our approach. 

MUSF is an integral inequality-based condition, which inherently introduces conservatism in the stability criteria. As noted in \cite{Liu2024nonconservative} and \cite{Wen2024stability}, less conservative criteria can even analyze the stability of random switching systems with all unstable subsystems. However, our results still require long enough sojourn time of stable subsystems to counterbalance the effects of unstable subsystems, and thus is conservative. In future work, we aim to reduce this conservatism and extend the stability analysis to random switching time-varying systems with all unstable subsystems.


%

%


\section*{Acknowledgment}
The author is grateful to the anonymous reviewers and the associate editor for their valuable and detailed comments that substantially helped for improving the quality of the paper.

\ifCLASSOPTIONcaptionsoff
  \newpage
\fi



\vspace{-0.3cm}
\bibliographystyle{IEEEtran}
\bibliography{mybib} 
\end{document}